\documentclass[11pt,oneside,reqno]{article}

\usepackage{amsmath,amsthm,amssymb,color,bm}

\usepackage{array}
\usepackage[margin=2.5cm,a4paper]{geometry}
\usepackage{bbm}
\usepackage{paralist}
\usepackage{mathrsfs}
\usepackage[breaklinks=true]{hyperref}

\usepackage[square]{natbib} 

\renewcommand{\cite}{\citep*}

\numberwithin{equation}{section}
\allowdisplaybreaks[4]

\theoremstyle{plain}
\newtheorem{theorem}{Theorem}[section]

\newtheorem{lemma}[theorem]{Lemma}

\theoremstyle{definition}
\newtheorem{definition}[theorem]{Definition}
\newtheorem{example}[theorem]{Example}
\newtheorem{remark}[theorem]{Remark}

\makeatletter

\renewcommand{\phi}{\varphi}
\newcommand{\eps}{\varepsilon}
\newcommand{\eq}{\eqref}

\newcommand{\lito}{\mathrm{o}}

\newcommand{\ind}{\mathbf{1}}

\newcommand{\Exp}{\mathop{\mathrm{Exp}}}

\newcommand{\N}{\mathrm{N}}

\def\E{\mathbbm{E}}

\newcounter{ctr}\loop\stepcounter{ctr}\edef\X{\@Alph\c@ctr}%
	\expandafter\edef\csname s\X\endcsname{\noexpand\mathscr{\X}}
	\expandafter\edef\csname c\X\endcsname{\noexpand\mathcal{\X}}
	\expandafter\edef\csname b\X\endcsname{\noexpand\boldsymbol{\X}}
	\expandafter\edef\csname I\X\endcsname{\noexpand\mathbbm{\X}}
	\expandafter\edef\csname r\X\endcsname{\noexpand\mathrm{\X}}
\ifnum\thectr<26\repeat

\def\be#1{\begin{equation*}#1\end{equation*}}
\def\ben#1{\begin{equation}#1\end{equation}}

\def\ba#1{\begin{align*}#1\end{align*}}
\def\ban#1{\begin{align}#1\end{align}}

\def\given{\typeout{Command 'given' should only be used within bracket command}}
\newcounter{@bracketlevel}
\def\@bracketfactory#1#2#3#4#5#6{
\expandafter\def\csname#1\endcsname##1{%
\addtocounter{@bracketlevel}{1}%
\global\expandafter\let\csname @middummy\alph{@bracketlevel}\endcsname\given%
\global\def\given{\mskip#5\csname#4\endcsname\vert\mskip#6}\csname#4l\endcsname#2##1\csname#4r\endcsname#3%
\global\expandafter\let\expandafter\given\csname @middummy\alph{@bracketlevel}\endcsname
\addtocounter{@bracketlevel}{-1}}%
}
\def\bracketfactory#1#2#3{%
\@bracketfactory{#1}{#2}{#3}{relax}{1mu plus 0.25mu minus 0.25mu}{0.6mu plus 0.15mu minus 0.15mu}
\@bracketfactory{b#1}{#2}{#3}{big}{1mu plus 0.25mu minus 0.25mu}{0.6mu plus 0.15mu minus 0.15mu}
\@bracketfactory{bb#1}{#2}{#3}{Big}{2.4mu plus 0.8mu minus 0.8mu}{1.8mu plus 0.6mu minus 0.6mu}
\@bracketfactory{bbb#1}{#2}{#3}{bigg}{3.2mu plus 1mu minus 1mu}{2.4mu plus 0.75mu minus 0.75mu}
\@bracketfactory{bbbb#1}{#2}{#3}{Bigg}{4mu plus 1mu minus 1mu}{3mu plus 0.75mu minus 0.75mu}
}
\bracketfactory{clc}{\lbrace}{\rbrace}
\bracketfactory{clr}{(}{)}
\bracketfactory{cls}{[}{]}

\def\abs#1{\vert#1\vert}

\def\bbabs#1{\Bigl\vert#1\Bigr\vert}

\def\angle#1{{\langle#1\rangle}}

\renewcommand\section{\@startsection {section}{1}{\z@}%
{-3.5ex \@plus -1ex \@minus -.2ex}%
{1.3ex \@plus.2ex}%
{\center\small\sc\mathversion{bold}\MakeUppercase}}

\def\subsection#1{\@startsection {subsection}{2}{0pt}%
{-3.5ex \@plus -1ex \@minus -.2ex}%
{1ex \@plus.2ex}%
{\bf\mathversion{bold}}{#1}}

\def\subsubsection#1{\@startsection{subsubsection}{3}{0pt}%
{\medskipamount}%
{-10pt}%
{\normalsize\itshape}{\kern-2.2ex. #1.}}

\def\blfootnote{\xdef\@thefnmark{}\@footnotetext}

\makeatother

\newcommand\G{\Gamma}

\def\l{\lambda}
\newcommand\ta{\textnormal{\textbf{a}}}

\newcommand\tZ{\textnormal{\textbf{Z}}}

\newcommand\tx{\textnormal{\textbf{x}}}

\newcommand{\simp}{\Delta_{K}}

\newcommand{\PD}{\textrm{PD}}
\newcommand{\DP}{\textrm{DP}}

\def\d{\delta}

\def\a{\alpha}
\def\b{\beta}

\begin{document}

\title{\sc\bf\large\MakeUppercase{Stein's method and duality of Markov processes}}
\author{\sc Han~L.~Gan}
\date{\it Northwestern University}
\maketitle

\begin{abstract} 
One of the key ingredients to successfully apply Stein's method for distributional approximation are solutions to the Stein equations and their derivatives. Using Barbour's generator approach, one can solve for the solutions to the Stein equation in terms of the semi-group of a Markov process, which is typically a diffusion process if it is a continuous distribution. For an arbitrary diffusion it can a difficult task to evaluate the semi-group and its derivatives. In this paper, for polynomial test functions, instead of calculating the semi-group of a diffusion, via a duality argument, we instead utilise the semi-group of a much simpler Markov jump process. This approach yields a new method for explicitly solving for the solutions of Stein equations for diffusion processes. We present both the general idea of the approach and examples for both univariate and multivariate distributions.
\end{abstract}


\section{Introduction}
Stein's method using generator approach,  pioneered in~\cite{B88,B90}, involves characterising a target distribution as the stationary distribution of an (infinitesimal) generator of a stochastic process. Such a generator is then used as the Stein operator or identity. As an example, consider the following generator,
\ba{
\cA f(x) = f''(x) - xf'(x).
}
This is the generator of an Ornstein-Uhlenbeck process, and notably the stationary distribution of such a process is the standard normal distribution. The general idea of the generator method can be summarised as follows.

Let $\cA f$ be a generator, and assume that its associated stationary distribution $\pi$ exists and is unique. Then for any function $h$ from a suitably rich family of test functions $\cF$, let $f_h$ be the solution to 
\ban{
\cA f_h(x) = h(x) - \E h(Z),\label{eq:steineq}
}
where $Z \sim \pi$. For any random variable $W$, our aim is to bound $\abs{\E h(W) - \E h(Z)}$ for all $h \in \cF$. To achieve this, we substitute $x = W$ into Equation~\eq{eq:steineq}, and take expectations of both sides. The problem is now transferred to taking expectations of the generator $\cA f_h(W)$. It turns out that properties of the function $f_h$ are essential, in particular supremum bounds on the function and its derivatives. Under the right conditions, (for example if $\cA$ generates a Feller semi-group), it can be shown that for any $h \in \cF$,
\ban{
f_h(x) = -\int_0^\infty [\E h(Z_x(t)) - \E h(Z) ]dt, \label{eq:steinsol}
}
where $Z_x(t)$ is a stochastic process following generator $\cA$ such that $Z_x(0) = x$. The main difficulty with using such an approach is that to evaluate~\eq{eq:steinsol}, one typically needs to have a very strong understanding of the process $Z_x(t)$. Our approach in this paper is to bypass this problem by evaluating $\E h(Z_x(t))$ for a diffusion by instead calculating the semi-group for a much simpler Markov jump process via a duality argument. 

We begin by defining the canonical definition of duality of two Markov processes with respect to a duality function. 

\begin{definition}
Let $X(t)$ and $Y(t)$ be two Markov processes with state spaces $E_1$ and $E_2$, and $H$ be a (bounded measurable) function on $(E_1, E_2)$. Then $X(t)$ is \emph{dual to $Y(t)$ with respect to a function $H$} if for all $x \in E_1, y \in E_2$ and $t \geq 0$,
\ban{
\E_xH(X(t),y) = \E^yH(x,Y(t)),\label{eq:dualid}
}
where $\E_x$ indicates $X(0) = x$ and $\E^y$ indicates $Y(0) = y$.
\end{definition}
For a summary of Markov dual processes, see the survey article~\cite{JK14}. Note that the restriction that $H$ be bounded and measurable can be weakened, and the equality above takes a slightly different form, see~\cite[Section 4.4]{EK86} for full details. 

In this paper we will be focusing on using a form of moment duality for the convergence of moments of distributions. We note that it can be trivially extended to polynomial test functions. Convergence of $k$-th moments, is generally not a sufficient condition to ensure convergence in distribution. If in addition to convergence of all moments, the limiting distribution has finite support, or finite moment generating function in a neighborhood of 0, then this yields convergence in distribution~\cite{Billingsley13}. The examples we present will usually satisfy at least one of these two conditions. 

Our goal is to use a dual process to calculate the semi-group for test functions that correspond to moments of random variables. The corresponding duality function for moments is therefore functions $H$ of the form $H(x,y) = x^y$. It remains an interesting open problem whether other forms of duality can also yield fruitful results. For example a common form of duality known as Siegmund duality, takes the form
\ba{
\IP_x( X_0 \leq Y_t) = \IP^y(X_t \leq Y_0).
}
Such a duality would give an alternative approach to calculating bounds in Kolmogorov distance. 

We begin with a simple motivating example for the Ornstein-Uhlenbeck process and Normal approximation, to illustrate how we can utilise duality to explicitly calculate solutions to the Stein equation for moment test functions. Recall
\ba{
\cA f(x) = f''(x) - xf'(x),
}
as the generator of the Ornstein-Uhlenbeck process. Furthermore define the following
\ba{
\cB f(y) = y(y-1)(f(y-2) - f(y)),
}
as the generator of a Markov jump process on non-negative integers such that when it is in state $y$, it waits an exponential $y(y-1)$ amount of time before it jumps down in step sizes of 2, and this process continues until it is absorbed at $1$ or $0$. It can be shown that these two processes are dual to each other with respect to the moment duality function $H(x,y) = x^y$, see~\cite[Example~4.4.9]{EK86} for full details. Set $X(t)$ and $Y(t)$ to be processes with generators $\cA$ and $\cB$ respectively. Then it can be shown that
\ban{
\E[ X(t)^{Y(0)}] = \E \left[ X(0)^{Y(t)} \exp \left( \int_0^t (Y^2(u) - 2Y(u)) du \right) \right].\label{eq:normaldual}
}
This slightly more complicated form of duality compared to~\eq{eq:dualid} due to the fact that our dual function is unbounded. Recall the solution to the Stein equation~\eq{eq:steinsol} takes the form $f_h(x) = -\int_0^\infty [\E h(Z_x(t)) - \E h(Z) ]dt$, where $Z \sim \N(0,1)$. Set $h(x) = x^k$, $X(0) = x$, and $Y(0) = k$, then~\eq{eq:normaldual} gives us an alternative approach to evaluating $\E h(Z_x(t))$, 
\ban{
f_h(x) = -\int_0^\infty \left\{\E \left[ X(0)^{Y(t)} \exp \left( \int_0^t (Y^2(u) - 2Y(u)) du \right) \right] - \E h(Z) \right\}dt.
}
The primary advantage of such an approach is that our solution rather than being in terms of the Ornstein-Uhlenbeck diffusion, is now written purely in terms of a relatively simple Markov jump process.

We compute a simple example case where $k=3$ to illustrate our general approach and also to confirm that this yields the correct solution. Let $\tau_1$ be an exponential random variable with parameter 6.  Then,
\ba{
f_h(x) &=  -\int_0^\infty \left\{ \E \left[ X(0)^{Y(t)} \exp \left( \int_0^t (Y^2(u) - 2Y(u)) du \right) \right] - \E h(Z) \right\}dt\\
	&= - \E \int_0^{\tau_1} x^3 \exp \left( \int_0^t Y^2(u) - 2Y(u) du \right)dt- \E \int_{\tau_1}^\infty x \exp \left( \int_0^t Y^2(u) - 2Y(u) du \right)dt\\
	&= -\E \int_0^{\tau_1} x^3 \exp \left\{3t \right\} dt- \E \int_{\tau_1}^\infty x \exp \left( 3\tau_1 - (t-\tau_1) \right) dt\\
	&= - \frac{x^3}{3}\E (e^{3\tau_1} -1) - x \E e^{3\tau_1}\\
	&= - \frac{x^3}{3} - 2x.
}
For the Ornstein-Uhlenbeck process and Normal approximation, a closed form for the solution to the Stein equation is well known~\cite{B90}. We can thus use this to confirm that our approach above is successful. For a general function $h$, the solution to the Stein equation is
\ban{
f_h(x) = -\int_0^\infty \left\{\E h(e^{-s} x + \sqrt{1-e^{-2s}} Z) - \E h(Z)\right\} ds.\label{eq:normalsol}
}
Hence for $h(x) = x^3$,
\ba{
f_h(x) &= -\int_0^\infty \E(e^{-s} x + \sqrt{1-e^{-2s}}Z)^3 ds\\
	&= - \int_0^\infty x^3 e^{-3s} + 3 e^{-s}x (1-e^{-2s}) ds\\
	&= -\frac{x^3}{3} - 2x,
}
which is the desired answer. As the example indicates, by finding an appropriate dual process we can certainly explicitly solve for solutions to the Stein equation for certain test functions. The primary question then becomes, how does one find an appropriate dual process? The existence of dual processes is often unclear, let alone a systematic approach for finding such a process, as discussed in the introduction of~\cite{JK14}. To address this issue, in this paper we use a modified version of duality, that will be relatively easy to find. 

If $f(x) = x^k$, then the generator of the Ornstein-Uhlenbeck process can be rewritten in the following form,
\ba{
\cA f(x) &= f''(x) - xf'(x)\\
	& =k(k-1) x^{k-2}- kx^k\\
	& =k ( (k-1) x^{k-2} - x^k). 
}
This suggests an alternative duality with a Markov jump process on the space of \emph{functions}. The process begins with the function $x^k$, and at a rate of $k$ it transitions to the function $(k-1)x^{k-2}$. Then at a rate of $k-2$ it transitions to the function $ (k-1)(k-3)x^{k-4}$ and so forth until it reaches its absorbing state. Depending on whether $k$ is even or odd, this process will either be absorbed at $(k-1)!!$ or 0 respectively, which importantly are the moments of the standard Normal distribution. Note if $f(x) = \a x$, $\cA f(x) = -\a x = 1(0-\a x)$, so can be interpreted as transitioning from $\a x$ to 0 at a rate of 1. Denote by $Y_x(1,t)$ to be the process on functions as described, such that $Y_x(1,0) = x^k$. (The argument 1 corresponds to the coefficient of $x^k$ in our test function.) We now continue our simple example where $k=3$. 

Suppose the following duality equation were true: Let $h(x) = x^3$, and $Z_x(t)$ be the Ornstein-Uhlenbeck process started at $x$, and
\ban{
\E h(Z_x(t)) = \E (Y_x(1,t)).\label{eq:dualex}
}
Then,
\ba{
f_h(x) = -\int_0^\infty \E h(Z_x(t)) &= -\int_0^\infty \E (Y_x(1,t)) dt\\
	&= -\E \int_0^{\tau_1} x^3 dt - \E \int_{\tau_1}^{\tau_2} 2x dt\\
	&= -x^3/3 - 2x,
}
where we used the fact that $\tau_1 \sim \Exp(3)$ and $\tau_2 - \tau_1 \sim \Exp(1)$. 

Our approach in this paper will be to use dual processes of the above form. To successfully apply Stein's method, one typically requires uniform bounds on the supremum of not only the solutions to the Stein equation, but also their derivatives. If the support of our target distribution is unbounded, our approach will explicitly solve for the Stein solution, but will usually not yield uniform supremum bounds. However, these bounds can still be used if more care is taken, see for example~\cite{Gaunt15}. In the case where the support of the distribution is finite, then one can use our approach to find uniform bounds on derivatives of the Stein equation. We shall see this in a few examples.

The remainder of the paper will be structured as follows. In Section 2, we will focus on univariate distributions, formally define our notion of duality, justify its construction and give the general approach for finding a suitable dual process. It will include various examples, and their results will mostly coincide exactly with known results derived from alternative methods. Section 3 will be analogous to Section 2, but for multivariate distributions. The final section will conclude with a discussion of the results and further open questions.

\section{Univariate solutions to Stein equations from duality}
Stein's method for univariate distributions has a rich and well developed theory, see for example the survey articles~\cite{LRS17,Ross11}. In this section we will focus on solutions to Stein equations for univariate distributions using a dual process. We begin by briefly explaining intuition behind the main idea of the duality. The generator of a Markov process has two arguments, a test function $f$ and a point/position in the sample space $x$. Loosely speaking, typically one interprets a generator by first fixing the test function $f$, and the generator in a sense describes how $x$ evolves in time. For our dual process, we reverse this interpretation, we instead fix $x$ and study how the test function $f$ evolves in time. Hence if the generator can be interpreted as or written in the form of a Markov jump process on functions, we can consider the dual process to be this jump process on functions. Now due to the fact that the semi-group of a Markov process can be characterised by its generator, given both the original and dual processes have the same generator, they must also have the same semi-group. We will confirm this fact by comparing our dual process solutions to known solutions using the original process when possible. The following lemma formalises the above argument.
\begin{lemma}\label{mainlemma}
Let $\cA$ be a generator of a one dimensional diffusion process and let $\cD(\cA)$ be its domain. Suppose $\cA$ generates a Feller semi-group on $\cD(\cA)$ and has unique stationary distribution $\mu$. If for $h(x) = x^k \in \cD(\cA)$ where $n \in \{ 1,2,\ldots\}$, and there exist constants $\{a_i\}, \{b_i\}$, such that $\cA f(x)$ can be written as
\ban{
\cA f(x) = \sum_{i=1}^k a_i [ b_i x^{k-i} - x^k],
}
then there exists a Markov jump process $Y(x,t)$ that is dual to $X(x,t)$ where $X(x,t)$ if a diffusion process with generator $\cA$ and $X(x,0) = x$, such that the solution $f_h(x)$ to the Stein equation~\eq{eq:steineq} satisfies
\ba{
f_h(x) = - \int_0^\infty \E Y_x(1,t) - E h(Z) dt,
}
where $Z \sim \mu$, and $Y_x(1,t)$ is a Markov jump process on functions of the form $g_x(a,k) = a x^k$ where $Y_x(1,0) = x^k$ and generator
\ba{
\cB x(g(a,k)) = \sum_{i=1}^k a_i [ g_x(a b_i,k-i) - g_x(a,k)].
}
\end{lemma}
\begin{proof}
The condition that we require a Feller semi-group ensures the existence of the solution to the Stein equation~\cite[Proposition~1.5(a)]{EK86}. Essentially it remains to show that the two semi-groups of $\cA$ and $\cB$ are the same, that is for all $t > 0$, $\E h(X_x(t)) = \E Y_x(1,t)$. \cite[Proposition~1.5~(c)]{EK86} yields
\ba{
\E h(X_x(t)) - h(x) &= \int_0^t \E [ \cA h(X_x(s)] ds\\
	&= \int_0^t \sum_{i=1}^k\E  a_i[ b_i X_x(s)^{k-i} - X_x(s)^k] ds,
}
hence
\ba{
\frac{d}{dt} \E h(X_x(t)) = \E \sum_{i=1}^k a_i[b_i X_x(t)^{k-i} - X_x(t)^k].
}
However it follows from the definition of $\cB$ and the equality of the two generators, that the process $Y_x(1,t)$ will satisfy exactly the same differential form. Given they both have the same initial value, the two semi-groups must be identical.
\end{proof}
%

\subsection{Univariate examples}
We first define some common notation which we will be using for all the examples. $\cA$ will denote the characterising diffusion operator, $Y_x(1,t)$ the corresponding dual process and $Z$ a random variable following the target distribution. Our test functions will always be of the form $h(x) = x^k$ unless otherwise specified. Note that this can easily be extended to polynomial test functions.
\begin{example}[Normal distribution]\label{ex:normal}
Recall from the introduction
\ba{
\cA f(x) &= f''(x) - xf'(x)\\
	& =k(k-1) x^{k-2}- kx^k\\
	& =k ( (k-1) x^{k-2} - x^k), 
}
and the description of the dual process therein. Note that $\E(Z^{2k}) =(2k-1)!!$ and $\E(Z^{2k+1}) = 0$. Then for $h(x) = x^{2k+1}$,
\ba{
f_h(x) &= -\int_0^\infty \E Y_x(1,t) dt\\
	&= - \E \sum_{i=0}^{k} \int_{\tau_i}^{\tau_{i+1}} \frac{(2k)!!}{(2k-2i)!!}x^{2k+1-2i} dt\\
	&= - \sum_{i=0}^k \frac{(2k)!!}{(2k-2i)!!(2k-2i+1)}x^{2k+1-2i}, 
}
where $\tau_0 = 0$ and $\tau_{i+1} - \tau_i \sim \Exp(2k-2i + 1)$. Similarly, for $h(x) = x^{2k}$,
\ba{
f_h(x) &= -\int_0^\infty \E [Y_x(1,t) - (2k-1)!!] dt\\
	&= - \E \sum_{i=0}^{k-1} \int_{\tau_i}^{\tau_{i+1}} \left[\frac{(2k-1)!!}{(2k-2i-1)!!}x^{2k-2i} - (2k-1)!!\right]dt\\
	&= - \sum_{i=0}^{k-1}\left[ \frac{(2k-1)!!}{(2k-2i-1)!!(2k-2i)}x^{2k-2i} - \frac{(2k-1)!!}{2k-2i}\right], 
}
where $\tau_0=0$ and $\tau_{i+1} - \tau_i \sim \Exp(2k-2i)$ and noting $Y_x(1,t) = (2k-1)!!$ for $t \geq \tau_k$. Routine calculations can confirm that these solutions coincide with~\eq{eq:normalsol}.
\end{example}

\begin{example}[Gamma distribution]
\cite{Luk94} developed Stein's method for the Gamma distribution using the generator method and we follow the same parameterisation. The Gamma distribution with shape parameter $r>0$ and scale parameter $\l>0$ has density function,
\ba{
f(x) = \frac{\lambda^r}{\G(r)} e^{-\l x}x^{r-1}, \ \ \ x > 0.
}
Furthermore if $X \sim \mathrm{Gamma}(r,\l)$, then $\E X^k = \frac{ \G(r+k)}{\lambda^k \G(r)}$. The generator with the Gamma distribution as its stationary distribution is
\ba{
\cA f(x) = x f''(x) + (r-\l x) f'(x).
}
Hence for $f(x) = x^k$,
\ba{
\cA f(x) &= k(k-1)x^{k-1} + rk x^{k-1} - \l k x^k\\
	&= \l k \left( \frac{ (k-1+r)}{\l}x^{k-1} - x^k \right).
}
\end{example}
The dual process can be therefore described as follows. Given the process is in state $a x^{k}$, it waits an $\Exp(\l n)$ distributed amount of time, before transitioning to $\frac{a( k-1 + r)}{\l} x^{k-1}$, and the process will be eventually absorbed at the constant $\frac{ \G(r+k)}{\lambda^k \G(r)}$.

Let $\tau_0 = 0$ and $\tau_{i+1} - \tau_i \sim \Exp(\l (k-i))$. Then for $h(x) = x^k$,
\ba{
f_h(x) &= -\int_0^\infty \left[ \E Y_x(1,t) - \frac{ \G(r+k)}{\lambda^k \G(r)} \right]dt\\
	&= -\E \sum_{j=0}^{k-1} \int_{\tau_j}^{\tau_{j+1}}\left[ \frac{\G(r+k)}{\l^j \G(r + k - j)} x^{k-j} - \frac{ \G(r+k)}{\lambda^k \G(r)}\right] dt\\
	&= - \sum_{j=0}^{k-1} \frac{\G(r+k)}{(k-j) \l^{j+1} \G(r + k - j)} x^{k-j} + \frac{ \G(r+k)}{\lambda^k \G(r)} \sum_{j=0}^{k-1} \frac{1}{\l(k-i)}.
}
\begin{remark}
\cite[Theorem~2.6]{Luk94} shows that for a smooth test function with bounded derivatives $h$, the $j$-th derivative of the solution to the Stein equation satisfies
\ba{
\| f_h^{(j)}\| \leq \frac{\|h^{(j)}\|}{j \lambda},
}
where $\| \cdot \|$ denotes the sup norm and $f^{(j)}$ denotes the $j$-th derivative. Our test functions do not have bounded derivatives, with exception of the case where $h(x) = x$ and for the first derivative. In which case our bounds yield $f_h^{(1)}(x) = \frac{1}{\l}$, so the two bounds are identical.  
\end{remark}
\begin{example}[Beta distribution]\label{ex:beta}
\cite{Dobler15, GR13, GRR17} developed Stein's method for the Beta$(\a,\b)$ distribution. The generator presented is as follows:
\ba{
\cA f(x) = x(1-x) f''(x) + (\a(1-x) - \b x) f'(x).
}
This is the generator of (one of the many forms of) a neutral Wright-Fisher diffusion. There are a variety of dual processes associated with various forms of Wright-Fisher diffusion, see for example~\cite{K82, G80, G16} for just a small sample. Our form of duality will be similar in flavour and analogous to the ancestral/looking backwards in time Kingman's coalescent. Set $f(x) = x^k$,
\ba{
\cA f(x) &= (x-x^2) k(k-1)x^{k-2} + (\a(1-x)  - \b x) kx^{k-1}\\
	&= k(k-1) (x^{k-1} - x^k) + k(\a + \b) \left(\frac{\a}{\a+\b} x^{k-1} - x^k\right).  
}
Our dual process can therefore be described in the following manner. For the function $h(x) = a x^k$, 
\begin{itemize}
\item At a rate of $k(k-1)$, $ax^k$ transitions to $ax^{k-1}$,
\item At a rate of $k(\a + \b)$, $ax^k$ transitions to $\frac{\a}{\a+\b} a x^{k-1}$.
\end{itemize}
Notably the net transition rates of this process ($k(k-1+\a+\b)$) are equivalent to Kingman's coalescent. 

Let $\tau_0 = 0$ and $\tau_{i+1} - \tau_i \sim \Exp((k-i)(k-i-1 + \a + \b))$. Then given $Y_x(1,0) = x^k$, 
\ba{
\E Y_x(1,\tau_1) = \frac{k-1}{k-1 + \a + \b} x^{k-1} + \frac{\a}{k-1+\a+\b} x^{k-1} = \frac{k-1 + \a}{k-1+\a+\b}x^{k-1}.
}
Similarly, for $1 \leq j \leq k$,
\ba{
\E Y_x(1,\tau_j) = \left(\prod_{i=1}^{j} \frac{k-i + \a}{k-i+\a+\b}\right)x^{k-j}.
}
Notably, $\E Y_x(1,\tau_k) = \prod_{r=0}^{k-1} \frac{\a + r}{\a + \b +r} = \E Z^k$, as expected and required.

\ba{
f_h(x) &= - \E \sum_{j=0}^{k-1} \int_{\tau_j}^{\tau_{j+1}} \left(\prod_{i=1}^{j} \frac{k-i + \a}{k-i+\a+\b}x^{k-j}  - \prod_{r=0}^{k-1} \frac{\a + r}{\a + \b +r}\right)dt\\
	&= - \sum_{j=0}^{k-1} \frac{1}{(k-j)(k-j-1+\a+\b)}  \left(\prod_{i=1}^{j} \frac{k-i + \a}{k-i+\a+\b}\right)x^{k-j} \\
	&\ \ \ + \left[\sum_{j=0}^{k-1} \frac{1}{(k-j)(k-j-1+\a+\b)}\right] \prod_{r=0}^{k-1} \frac{\a + r}{\a + \b +r}.
}
Given the Beta distribution has finite support on $[0,1]$ we can use the above solution to find uniform bounds for $f_h(x)$ and its derivatives. For example, noting that $x \leq 1$, the first derivative satisfies
\ban{
\abs{f_h'(x)} &=  \sum_{j=0}^{k-1} \frac{1}{(k-j)(k-j-1+\a+\b)}  \left(\prod_{i=1}^{j} \frac{k-i + \a}{k-i+\a+\b}\right)(k-j)x^{k-j-1}\label{eq:beta1}\\
	&\leq \sum_{j=0}^{k-1} \frac{1}{(k-j-1+\a+\b)}\notag\\
	&\leq \frac{k}{\a + \b}.\notag
}
Similar computations show for the $n$-th derivative,
\ba{
\abs{ f_h^{(n)}(x)} \leq \frac{k(k-1) \ldots (k-n+1)}{n(\alpha + \beta + n-1)}.
}
Notably, these bounds correspond exactly to the bounds in~\cite[Theorem 5]{GRR17}. The bounds in~\cite{GRR17} are more general, they essentially cover all smooth test functions, where as the approach in this paper will only cover polynomial test functions. However, given polynomials are dense we have not lost too much, and this proof is far shorter and less complicated than the proof presented in the aforementioned work which involves couplings and coalescent theory.
\end{example}

\section{Multivariate solutions to Stein equations from Duality}
Results for Stein's method for multivariate distributions remain relatively limited. \cite{BLX18, BLX18a}~provide a framework for multivariate discrete distributions that is applicable in many scenarios. Outside a few examples such as multivariate Normal and Dirichlet, Stein's method for multivariate continuous distributions is underdeveloped. The generator method is the primary approach multivariate setting as nothing fundamentally changes when considering multivariate Markov processes compared to the univariate case, whereas often more purely analytic approaches become intractable. In a similar vein, our dual process approach will be equally effective in the multivariate setting without too much additional complexity. We omit a multivariate version of Lemma~\ref{mainlemma} as it is completely analogous and not particularly enlightening, and we instead focus on examples. 
\subsection{Multivariate examples}
In the multivariate setting we need to study not only the convergence of the moments of the marginals, but also convergence of all mixed moments. Hence we will be considering test functions of the form $f(\tx) = \prod_{i=1}^p w_i^{k_i}$ for some $p$ and $k_i$.
\begin{example}[Multivariate Normal]
For $p$-dimensional Normal approximation with mean $0$ and covariances $\{\sigma_{ij}\}^p_{i,j=1}$, the $p$-dimensional Ornstein-Uhlenbeck generator can be written in the form
\ba{
\cA f( \tx) = \sum_{i=1}^p \sum_{j=1}^p \sigma_{ij}  \frac{\partial{^2}}{\partial w_j \partial w_i}f(\tx) - \sum_{i=1}^p w_i \frac{\partial}{\partial w_i}f(\tx).
}
For $f(\tx) = \prod_{i=1}^p w_i^{k_i}$, and since the (multivariate) marginal distributions of the Normal distribution are still (multivariate) Normally distributed, without loss of generality assume that $k_i \geq 1$ for all $i$, and denote $f^{\{ij\}}(\tx) = \frac{f(\tx)}{w_i w_j}$. Then,
\ba{
\cA f(\tx) &= \sum_{i=1}^p k_i \left[ \sum_{j=1}^p \sigma_{ij}(k_j - \delta_{ij}) f^{\{ij\}}(\tx) - f(\tx) \right]\\
	&= \sum_{i=1}^p k_i (k_i-1) [\sigma_{ii} f^{\{ii\}}(\tx) - f(\tx)] + \sum_{i=1}^p \sum_{j\neq i} k_i k_j [\sigma_{ij}  f^{\{ij\}}(\tx) - f(\tx)],
}
where $\delta_{ij}$ denotes the Kronecker delta function. Given the above form, the process can be described as follows. For all $1 \leq i \leq p$,
\begin{itemize}
\item At rate $k_i(k_i-1)$, $f(\tx)$ transitions to $\sigma_{ii} f^{\{ii\}}(\tx)$,
\item At rate $k_ik_j $, $f(\tx)$ transitions to $\sigma_{ij} f^{\{ij\}}(\tx)$.
\end{itemize}
Analogously to Example~\ref{ex:normal} we can then use this dual process to solve for $f_h(x)$. We omit the calculation for the solution to the Stein equation as it is similar to Example~\ref{ex:normal} but rather long and tedious, and furthermore there already exists a well known nice closed form solution,
\ba{
f_h(\tx) =  - \int_0^\infty \left[ \IE( h(e^{-s}\tx + \sqrt{1-e^{-2s}}\Sigma^{\frac12} \tZ) - \IE h(\Sigma^{\frac12}\tZ) \right]ds,
}
where $\tZ$ is a random vector having standard multivariate normal distribution of dimension $p$.
\end{example}

\begin{example}[Dirichlet distribution]
The Dirichlet distribution can be viewed as a multivariate extension of the Beta distribution. We follow the parameterisation and generator used in~\cite{GRR17}. The Dirichlet distribution with parameters~$\ta=(a_1,\dots,a_K)$, where~$a_1>0, \ldots, a_K>0$, is supported on the~$(K-1)$-dimensional open simplex, which we parameterise as
\be{
	\simp=\left\{\tx=(x_1,\ldots, x_{K-1}): x_1> 0, \ldots, x_{K-1} > 0,  \sum_{i=1}^{K-1} x_i < 1\right\}\subset \IR^{K-1}.
}
On~$\simp$, the Dirichlet distribution with parameter $\ta$ has density
\ben{
	\psi_\ta(x_1,\dots,x_{K-1}) = \frac{\Gamma(s)}{\prod_{i=1}^K \Gamma(a_i)} \prod_{i=1}^K x_i^{a_i-1}, 
}
where~$s=\sum_{i=1}^K a_i$, and where we set~$x_K=1-\sum_{i=1}^{K-1} x_i$, as we shall often do whenever considering vectors taking values in~$\Delta_{K}$. 

The Stein operator used is the generator of the neutral Wright-Fisher diffusion with parent independent mutation,
\ba{
\cA f(\tx) = \sum_{i,j=1}^{K-1} x_i(\delta_{ij} - x_j) \frac{\partial^2}{\partial x_i \partial x_j} f(\tx) + \sum_{i=1}^{K-1}(a_i - sx_i) \frac{\partial}{\partial x_i}f(\tx),
}
where $\d_{ij}$ denotes the Kronecker delta function. Similar the multivariate Normal distribution, the multivariate marginal distributions of the Dirichlet distribution are still Dirichlet distributed, hence we can without loss of generality assume $f(x) = \sum_{i=1}^{K-1} x_i^{k_i}$ where $k_i \geq 1$ for all $i$. As in the previous example set $f^{\{i\}}(\tx) = \frac{f(\tx)}{x_i}$. Then
\ba{
\cA f(\tx) &= \sum_{i,j=1}^{K-1} x_i(\d_{ij} - x_j) k_i(k_j-\delta_{ij})f^{\{ij\}}(\tx)+ \sum_{i=1}^{K-1} k_i \left(a_i f^{\{i\}}(\tx) - sf(\tx)\right)\\
	&= \sum_{i,j=1}^{K-1} k_i(k_j-\delta_{ij})\left(\delta_{ij} f^{\{j\}}(\tx) - f(\tx)\right) + \sum_{i=1}^{K-1} sk_i \left(\frac{ a_i}{s} f^{\{i\}}(\tx) - f(\tx)\right)\\
	&= \sum_{i=1}^{K-1} k_i(k_i-1) \left(f^{\{i\}}(\tx) - f(\tx)\right) + 2\sum_{1 \leq i < j \leq K-1} k_ik_j\left( 0 - f(\tx) \right)\\
	&\ \ \ \ \ \ \ \ \ \ + \sum_{i=1}^{K-1} sk_i \left(\frac{ a_i}{s} f^{\{i\}}(\tx) - f(\tx)\right).
}
\end{example}
The dual process is unsurprisingly similar to the dual process in Example~\ref{ex:beta}, with exception of an additional `interaction' term in its generator. We describe the process as follows: For all $i,j$,
\begin{itemize}
\item At rate $k_i(k_i-1)$, $f(\tx)$ transitions to $f^{\{i\}}(\tx)$,
\item At rate $sk_i$, $f(\tx)$ transitions to $\frac{a_i}{s}f^{\{i\}}(\tx)$,
\item At rate $k_ik_j$, $f(\tx)$ transitions to 0.
\end{itemize}
For this example we will take a different approach, rather than explicitly solving for the solution to the Stein equation and then taking derivatives of the solution, we will instead use a coupling argument. Furthermore, we will only focus on bounding the derivative of the solution to the Stein equation rather than the solution itself, as derivative bounds are typically what are required in applications. Let $\eps_i$ be the unit vector in the $i$-th direction, then without loss of generality set $i=1$. 
\ba{
\frac{\partial}{\partial x_1} f_h(\tx) &= \lim_{\eps \to 0} \frac1\eps ( f_h(\tx + \eps e_1) - f_h(\tx))\\
	&= -\lim_{\eps \to 0} \frac1\eps \int_0^\infty [\E h(Z_{\tx + \eps e_1}(t)) - \E h(Z_{\tx}(t)) ]dt\\
	&= -\lim_{\eps \to 0} \frac1\eps \int_0^\infty [\E Y_{\tx + \eps e_1}(1,t) - \E Y_{\tx}(1,t) ]dt
}
To bound this quantity, we first observe that for $\tx \in \simp$,
\ba{
\abs{(x_1 + \eps)^{k_1}x_2^{k_2} \ldots x_{K-1}^{k_{K-1}} - x_1^{k_1}x_2^{k_2} \ldots x_{K-1}^{k_{K-1}}} \leq k_1 \eps + \lito(\eps).
}
We couple the transition times of the two processes $Y_{\tx + \eps e_1}(1,t)$ and $Y_{\tx}(1,t)$ exactly, and we can hence focus predominantly upon only the $x_1$ term. Furthermore it suffices to ignore the third class of transition where $f(\tx)$ transitions to 0 as an upper bound as this will only increase the coupling time of the processes $(Y_{\tx + \eps e_1}(1,t), Y_{\tx}(1,t)$), and we apply this to avoid interactions between the different dimensions in our process. Given we only focus on the one dimension, we can ignore all the transitions that involve the other dimensions, and use an identical calculation to~\eq{eq:beta1},
\ba{
\bbabs{\frac{\partial}{\partial x_1}f_h(\tx)} &\leq \sum_{j=0}^{k_1-1} \frac{1}{(k_1-j)(k_1-j-1+s)}  \left(\prod_{i=1}^{j} \frac{k_1-i + a_1}{k_1-i+s}\right)(k_1-j)\\
	&\leq \sum_{j=0}^{k_1-1} \frac{1}{(k_1-j-1+s)} \leq \frac{k_1}{s}.
}
Hence
\ba{
\sup_{\tx \in \simp}\sup_{1 \leq i \leq K-1} \bbabs{\frac{\partial}{\partial x_1}f_h(\tx)} \leq \frac{ \max_{1 \leq i \leq K-1} k_i} {s}. 
}
As in Example~\ref{ex:beta} this bound corresponds to the bounds in~\cite{GRR17}. We leave the higher order derivatives as an exercise for interested readers. Interestingly this suggests that there may exist an opportunity to improve upon these bounds given we were slightly crude in ignoring the transitions of the type $f(\tx) \to 0$.

\begin{example}[Poisson-Dirichlet distribution and Dirichlet process]\label{ex:PD}
Our final example is an application for an infinite dimensional distribution. The original motivation for this work was for the formulation of Stein's method for Poisson-Dirichlet and Dirichlet process approximation. For the following, full details can be found in a forthcoming preprint~\cite{GR19}.

We begin by describing the Poisson-Dirichlet distribution. Define a random interval partition of the interval $[0,1]$. We encode the space of interval partitions as
\be{
\nabla_\infty:=\bclc{(p_1,p_2,\ldots): p_1\geq p_2 \geq \cdots, \sum_{i=1}^\infty p_i=1},
}
and define~$\PD(\theta)$, the \emph{Poisson-Dirichlet} distribution on $\nabla_\infty$ with parameter $\theta>0$. The Poisson-Dirichlet distribution can be described/constructed in a 
number of ways, such as the limiting distribution of proportions in integer partitions, the limit as the dimension goes to infinity, of appropriate Dirichlet distributions, the distribution of the normalised points of a Poisson process with intensity measure $\theta e^{-x}/x, x > 0$~\cite{Kingman75} and via stick breaking schemes and the GEM distribution; see \cite{Pitman06} for an overview. 

We define a Dirichlet process with parameters $\theta>0$ and $\pi$ a probability measure on a compact metric space~$E$, denoted $\DP(\theta, \pi)$, to be the random element of $M_1(E)$, the space of the probabilities measures on~$E$ equipped with the weak topology. We say $\mu\sim\DP(\theta, \pi)$ if 
\ben{
\mu=\sum_{i=1}^\infty P_i \delta_{\xi_i}, \label{mucomp}
}
where $(P_1,P_2,\ldots)\sim \PD(\theta)$ is independent of $\xi_1,\xi_2,\ldots$, which are i.i.d.\ with distribution~$\pi$.
Thus we think of $\DP(\theta,\pi)$ as the Poisson-Dirichlet distribution with independent $\pi$-distributed labels. 

The Stein operator used for Dirichlet process approximation is a particular class of Fleming-Viot processes. Define an operator~$A$ by
\begin{align}
A \phi(x) = \frac\theta2 \int_E (\phi(y) - \phi(x))\pi(dy), \label{A}
\end{align}
where the domain of $A$ is the set of continuous functions $C(E)$. We define the generator~$\cA$ of our Fleming-Viot Markov process; for $F\in\cD$
and $\mu\in M_1(E)$, 
\ben{\label{eq:dpgen}
\cA F (\mu)=\int_{E} \left(A \frac{\partial F(\mu)}{\partial \d_x}\right) \mu(dx) + 
	\frac{1}{2} \int_{E^2} (\mu(dx)\delta_x (dy) - \mu(dx) \mu(dy)) \frac{\partial^2 F(\mu)}{\partial \d_x \partial \d_y},
}
where $\delta_x(\cdot)$ is Dirac measure, and for any measure $\nu \in M_1$ we define the directional (G\^ateaux) derivatives by, 
\ban{
\frac{\partial F(\mu)}{\partial \nu}
	&:=\lim_{\eps \to 0^+} \frac{F((1-\eps)\mu+ \eps \nu) - F(\mu)}{\eps},
}
with higher order derivatives analogously defined.

For $f : E \mapsto \IR$, let $\angle{f,\mu}$ denote the expectation of $f$ with respect to $\mu$, that is $\angle{f,\mu} = \int_E f(x) \mu(dx)$. Furthermore, denote let $\mu^k$ denote the product measure of $k$ independent copies of $\mu$, and for  $f : E^k \mapsto \IR$ analogously define $\angle{f, \mu^k}$. The domain of the generator~\eq{eq:dpgen} we wish to consider is
\ben{
\cD = \{ F(\mu) := \angle{\phi, \mu^k}: k \in \IN, \phi \in \rB(E^k)\},\label{SD2}
}
where $\rB(E^{k})$ is the space of all bounded functions on $E^k$. To see how this space of test functions is related to moments, for example set $\phi(x,y) = \ind_{x=y}$, then $F(\mu) = \sum_{i=1}^\infty P_i^2$, where the $P_i$ are as in~\eq{mucomp}.

For $F = \angle{\phi,\mu^k}$, simple computations show that  our generator~\eq{eq:dpgen} can be reduced to the form
\ban{
\cA F(\mu) = \sum_{1 \leq i < j \leq k} [\angle{\Phi_{ij}^{(k)}\phi, \mu^{k-1}} - \angle{\phi, \mu^k}] + \frac{\theta}{2} \sum_{1 \leq i \leq k}[\angle{\phi, \mu^{i-1} \pi \mu^{k-i}} - \angle{\phi,\mu^k}],\label{eq:dualgen}
}
where $\Phi_{ij}^{(k)}\phi(x_1, \ldots, x_{k-1}) = \phi(x_1, \ldots, x_{j-1}, x_i, x_j, \ldots, x_{k-1})$. From the above formulation, we follow \cite[Section~3]{EK93} (in turn following \cite{DH82}), and define our dual
Markov jump process on $\cup_{j\geq1} \rB(E^j)$ with
$Y(0)=\phi$.
Informally, there are two types of transition for our jump process on functions 
\begin{itemize}
\item At a rate of 1, for each $1 \leq i,j \leq k$, the argument $x_i$ in $\phi(\tx)$ is replaced with $x_j$,
\item At a rate of $\frac{\theta}{2}$, for each $1 \leq i \leq k$, the argument $x_i$ in $\phi(\tx)$ is replaced with a $\pi$ distributed random variable.
\end{itemize}
As stated earlier, for full details about this formulation and how to use this to calculate bounds for derivatives of the Stein solution, see~\cite{GR19}. 
\end{example}

\section{Discussion}
As seen by the numerous examples, this duality approach appears to give results that are on par with existing techniques for calculating Stein solutions. To the author's knowledge, these are the first results where such a duality is used to calculate solutions to the Stein equation. We believe that there is potential for the ideas presented to be pushed and extended in a variety of interesting directions. In particular we present a few open questions that are of interest to the author.
\begin{enumerate}
\item Our focus in this paper was on test functions for moments. Are there other classes of test functions that can be approached using a dual process? For example is it possible to find a duality argument for test functions that correspond with Kolmogorov or Wasserstein distances?
\item In this paper, we convert difficult problems on diffusions to simpler problems on Markov jump processes. Is it possible to reverse this? \cite{BX01} give a flexible framework for Stein's method for Markov birth-death jump processes that is widely applicable, but there are examples such as the compound Poisson distribution that have multiple births and/or death events that remain intractable. Is it possible that problems involving complicated Markov jump processes could be instead approached using a simpler diffusion or perhaps even a simpler jump process?
\item The key term in our Stein operators that enables the duality argument is the $x f'(x)$ since this yields $xf'(x) = k f(x)$ for $f(x) = x^k$. This is therefore similar to size bias distributions. It would be interesting to explore if there is a direct connection between size biasing and moment duality.
\item Not all Stein operators have the key $xf'(x)$ term, for example the generalised Gamma distribution~\cite{PRR16} has Stein operator $\cA f(x) = f''(x) - \left( \frac{\alpha - 1}{x} - \b x^{\b-1}\right) f'(x)$, for some $\a, \b > 0$. Does there still exist some form of moment duality outside of the case $\b = 1$? Does there still exist a form of moment duality when we do not have an $xf'(x)$ term?
\end{enumerate}

\section{Acknowledgments}
The author would like to thank Nathan Ross for his contributions to Example~\ref{ex:PD} and for helpful discussions.

\bibliographystyle{apalike}
\bibliography{references.bib}

\end{document}